\PassOptionsToPackage{usenames,dvipsnames}{xcolor}
\documentclass[a4paper]{amsart}

\usepackage{amsfonts, amsmath, amssymb, pifont}
\usepackage{enumitem}
\usepackage{upref}
\usepackage[all]{xy}
\usepackage{hyperref}
\usepackage[capitalise]{cleveref}
\usepackage{mathtools}

\usepackage{todonotes}

\usepackage{tikz}
\colorlet{dgrmred}{red!60!white}
\colorlet{dgrmblu}{blue!60!white}
\colorlet{dgrmpur}{Plum!60!white}
\colorlet{dgrmmag}{magenta!60!white}
\colorlet{dgrmgrn}{green!60!white}
\colorlet{dgrmyel}{yellow!60!white}

\usepackage[style=numeric,
            sorting=nyt,
	        backend=bibtex,
			url=false,
			doi=false,
			isbn=false,
			giveninits=true]{biblatex}

\renewbibmacro{in:}{%
	\ifentrytype{article}{}{\printtext{\bibstring{in}\space}}}

\DeclareFieldFormat
[misc,article,inbook,incollection,inproceedings,patent,thesis,unpublished]
{title}{#1\isdot}

\renewbibmacro{volume+number+eid}{%
	\iffieldundef{number}{}{\printtext[parens]{%
			\printfield{number}}}
	\printfield[bold]{volume}
	\setunit{\space}%
	\printfield{eid}}

\renewbibmacro*{series+number}{%
	\printfield{series}%
	\setunit*{\addspace}%
	\printfield{number}%
	\newunit\nopunct}

\DeclareFieldFormat[article]{pages}{#1}

\renewbibmacro*{publisher+location+date}{%
	\printlist{location}%
	\iflistundef{publisher}
	{\setunit*{\addcomma\space}}
	{\setunit*{\addcolon\space}}%
	\printtext[parens]{\printlist{publisher}%
		\setunit*{\addcomma\space}%
		\usebibmacro{date}}
	\newunit}

\renewbibmacro*{location+date}{%
	\printlist{location}%
	\setunit*{\space}%
	\printtext[parens]{\usebibmacro{date}}%
	\newunit}

\NewBibliographyString{submitted}
\DefineBibliographyStrings{english}{%
  submitted = {submitted},
}

\theoremstyle{plain}
\newtheorem{thm}{Theorem}[section]
\newtheorem*{thm*}{Theorem}
\newtheorem{thmletter}{Theorem}[section]

\newtheorem{lem}[thm]{Lemma}

\newtheorem*{conj*}{Conjecture}
\newtheorem{cor}[thm]{Corollary}
\newtheorem*{cor*}{Corollary}
\newtheorem{corletter}[thmletter]{Corollary}

\newtheorem{prop}[thm]{Proposition}

\theoremstyle{definition}
\newtheorem{defn}[thm]{Definition}

\theoremstyle{remark}
\newtheorem{rem}[thm]{Remark}

\newcommand{\defnemph}[1]{\emph{#1}}
\newcommand{\iso}{\cong}
\newcommand{\NN}{\mathbb{N}}

\newcommand{\ZZ}{\mathbb{Z}}
\newcommand{\QQ}{\mathbb{Q}}

\newcommand{\CC}{\mathbb{C}}
\newcommand{\field}{\Bbbk}

\newcommand{\TL}{\mathrm{TL}}
\newcommand{\JW}{\mathrm{JW}}
\newcommand{\sred}{{\color{red} s}}
\newcommand{\tblu}{{\color{blue} t}}
\newcommand{\upur}{{\color{Plum} u}}
\newcommand{\vgrn}{\color{OliveGreen} v}
\newcommand{\qbinom}[2]{\genfrac{[}{]}{0pt}{}{#1}{#2}}

\DeclareMathOperator{\Hom}{Hom}

\DeclareMathOperator{\Frac}{Frac}

\DeclareMathOperator{\pTr}{pTr}

\DeclareMathOperator*{\coeff}{coeff}

\begin{document}
\title[Two-colored Jones--Wenzl projectors]{Existence and rotatability of the two-colored Jones--Wenzl projector}
\author{Amit Hazi}
\address{Department of Mathematics\\
      University of York\\
      York\\ 
      YO10 5DD\\
      United Kingdom}
\email{amit.hazi@york.ac.uk}

\begin{abstract}
The two-colored Temperley--Lieb algebra $2\TL_R(\prescript{}{s}{n})$ is a generalization of the Temperley--Lieb algebra. 
The analogous two-colored Jones--Wenzl projector $\JW_R(\prescript{}{s}{n}) \in 2\TL_R(\prescript{}{s}{n})$ plays an important role in the Elias--Williamson construction of the diagrammatic Hecke category.
We give conditions for the existence and rotatability of $\JW_R(\prescript{}{s}{n})$ in terms of the invertibility and vanishing of certain two-colored quantum binomial coefficients.
As a consequence, we prove that Abe's category of Soergel bimodules is equivalent to the diagrammatic Hecke category in complete generality.
\end{abstract}


\maketitle

\section{Introduction} \label{intro}

Let $R$ be a commutative ring, and fix two scalars $[2]_s,[2]_t \in R$.
The \defnemph{two-colored Temperley--Lieb algebra} $2\TL_R(\prescript{}{s}{n}):=2\TL_R(\prescript{}{s}{n};[2]_s,[2]_t)$ is the $R$-algebra with generators 
$e_i$ for $1 \leq i \leq n-1$,
subject to the relations
\begin{align}
e_i^2& =-[2]_s & & \text{$i$ odd,} \label{eq:oddquadratic} \\
e_i^2& =-[2]_t & & \text{$i$ even,} \label{eq:evenquadratic} \\
e_i e_j & =e_j e_i & & \text{for $|i-j|>1$,} \\
e_i e_{i\pm 1} e_i & =e_i & & \label{eq:braidreln}
\end{align}
The algebra $2\TL_R(\prescript{}{t}{n})$ is defined identically, except that the parity conditions on the relations \eqref{eq:oddquadratic} and \eqref{eq:evenquadratic} are swapped. 
These algebras (introduced by Elias in \cite{elias-dihedralcathedral}) form a generalization of the ordinary Temperley--Lieb algebra, which occurs as a special case when $[2]_s=[2]_t$. 
By a standard argument there is an $R$-basis of $2\TL_R(\prescript{}{s}{n})$ consisting of monomials in the generators $e_i$.

We call a non-zero idempotent $\JW_R(\prescript{}{s}{n}) \in 2\TL_R(\prescript{}{s}{n})$ (and similarly for $\prescript{}{t}{n}$) a \defnemph{two-colored Jones--Wenzl projector} if $e_i\JW_R(\prescript{}{s}{n})=0$ for all $1 \leq i \leq n-1$ and the coefficient of $1$ in $\JW_R(\prescript{}{s}{n})$ is $1$. 
Such idempotents (if they exist) are unique. 


The behavior of $2\TL_R(\prescript{}{s}{n})$ is controlled by certain elements $[n]_s,[n]_t \in R$ for $n \in \ZZ$ called the \defnemph{two-colored quantum numbers}. 
These elements (defined in \eqref{eq:twocolqnum}) are bivariate polynomials in $[2]_s$ and $[2]_t$ which are analogous to ordinary quantum numbers.
For an integer $0 \leq k \leq n$ the \defnemph{two-colored quantum binomial coefficient}
\begin{equation*}
\qbinom{n}{k}_{s}=\frac{[n]_{s}!}{[k]_{s}![n-k]_{s}!}=\frac{[n]_{s}[n-1]_{s}\dotsm [n-k+1]_{s}}{[k]_{s}[k-1]_{s} \dotsm [1]_{s}}
\end{equation*}
can also be shown to be an element of $R$.
Our first main result is the two-colored analogue of the well-known existence theorem for ordinary Jones--Wenzl projectors.

\begin{thmletter} \label{existence}
The two-colored Jones--Wenzl projector $\JW_R(\prescript{}{s}{n})$ exists if and only if $\qbinom{n}{k}_{s}$ is invertible in $R$ for each integer $0 \leq k \leq n$.
\end{thmletter}


The terminology for two-colored Temperley--Lieb algebras comes from their presentation as diagram algebras.
We associate the labels $s$ and $t$ with the colors red and blue, respectively, writing $\sred$ and $\tblu$ for emphasis.
A \defnemph{two-colored Temperley Lieb diagram} is a Temperley--Lieb diagram with the planar regions between strands colored with alternating colors. 
As a diagram algebra $2\TL_R(\prescript{}{\sred}{n})$ is spanned by two-colored Temperley--Lieb diagrams with $n$ boundary points on the top and bottom whose leftmost region is colored red.
A blue disk inside a red region evaluates to $-[2]_{\sred}$, while a red disk inside a blue region evaluates to $-[2]_{\tblu}$. 
Moreover we only consider two-colored Temperley--Lieb diagrams up to isotopy.
These diagrammatic relations directly correspond to \eqref{eq:oddquadratic}--\eqref{eq:braidreln}.
We draw the two-colored Jones--Wenzl projector as a rectangle labeled $\JW_R(\prescript{}{\sred}{n})$:
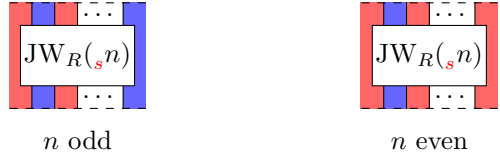
\begin{equation*}
\begin{gathered}
        \begin{tikzpicture}[xscale=-.3,yscale=0.2]
          \begin{scope}
            \clip (-3,3.5) rectangle (3,-3.5);
            \draw[fill=dgrmred] (0,4) rectangle (4,-4);
            \draw[fill=dgrmblu] (1,4) rectangle (2,-4);
         \draw[fill=dgrmblu] (-4,4) rectangle (-2,-4);
         \draw[fill=white] (-2.5,2) rectangle (2.5,-2);
         \node at (0,0) {$\JW_R(\prescript{}{\sred}{n})$};
         \node at (-1,2.7) {$\dots$};
         \node at (-1,-2.7) {$\dots$};
       \end{scope}
       \draw[dashed] (-3,3.5) to (3,3.5);
       \draw[dashed] (-3,-3.5) to (3,-3.5);
     \end{tikzpicture} \\
\text{$n$ odd}
\end{gathered} \qquad \qquad \qquad \qquad 
\begin{gathered}
        \begin{tikzpicture}[xscale=-.3,yscale=0.2]
          \begin{scope}
            \clip (-3,3.5) rectangle (3,-3.5);
            \draw[fill=dgrmred] (0,4) rectangle (4,-4);
            \draw[fill=dgrmblu] (1,4) rectangle (2,-4);
         \draw[fill=dgrmred] (-4,4) rectangle (-2,-4);
         \draw[fill=white] (-2.5,2) rectangle (2.5,-2);
         \node at (0,0) {$\JW_R(\prescript{}{\sred}{n})$};
         \node at (-1,2.7) {$\dots$};
         \node at (-1,-2.7) {$\dots$};
       \end{scope}
       \draw[dashed] (-3,3.5) to (3,3.5);
       \draw[dashed] (-3,-3.5) to (3,-3.5);
     \end{tikzpicture} \\
\text{$n$ even}
\end{gathered}
\end{equation*}

Suppose both $\JW_R(\prescript{}{\sred}{n})$ and $\JW_R(\prescript{}{\tblu}{n})$ exist.
We say that $\JW_R(\prescript{}{\sred}{n})$ is \defnemph{rotatable} if the (clockwise and counterclockwise) rotations of $\JW_R(\prescript{}{\sred}{n})$ by one strand are equal to some scalar multiple of $\JW_R(\prescript{}{\tblu}{n})$:
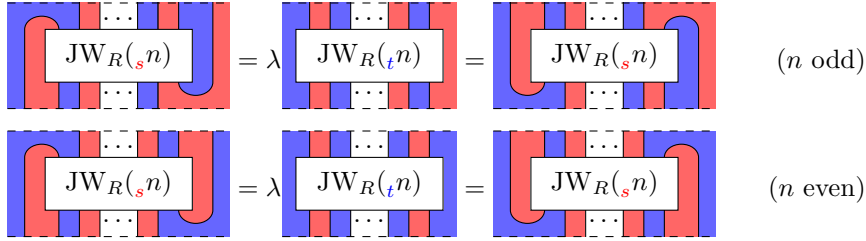
\begin{align*}
\begin{gathered}
  \begin{tikzpicture}[xscale=.45,yscale=.35]
    \begin{scope}
         \clip (-3,2) rectangle (3.5,-2);
    \draw[fill=dgrmblu] (-3.3,3) rectangle (1.4,-3);
    \draw[fill=dgrmred] (3.8,3) rectangle (1.4,-3);
             \draw[fill=dgrmred] (-.9,3) rectangle (0.5,-3);
             \draw[fill=white] (-.3,3) rectangle (0.8,-3);
             \draw[fill=dgrmblu] (2,3) to (2,-1)
         to[out=-90,in=180] (2.5,-1.5) to[out=0,in=-90] (3,-1) to
         (3,3) to (2,3);
         \draw[fill=dgrmred] (-1.5,-3) to (-1.5,1)
         to[out=90,in=0] (-2,1.5) to[out=180,in=90] (-2.5,1) to
         (-2.5,-3) to (-1.5,-3);
         \draw[fill=white] (-1.9,1) rectangle (2.4,-1);
         \node at (.25,0) {$\JW_R(\prescript{}{\sred}{n})$};
         \node at (.3,1.4) {$\dots$}; 
         \node at (.3,-1.4) {$\dots$};
         \end{scope}
         \draw[dashed] (-3,2) to (3.5,2);
         \draw[dashed] (-3,-2) to (3.5,-2);
       \end{tikzpicture}
\end{gathered}
    & =
     \lambda
\begin{gathered}
  \begin{tikzpicture}[xscale=.45,yscale=.35]
    \begin{scope}
         \clip (-2.3,2) rectangle (2.8,-2);
    \draw[fill=dgrmblu] (-3.3,3) rectangle (.5,-3);
             \draw[fill=dgrmred] (-1.5,3) rectangle (-.9,-3);
    \draw[fill=dgrmred] (3.8,3) rectangle (.5,-3);
    \draw[fill=dgrmblu] (2,3) rectangle (1.4,-3);
             \draw[fill=white] (-.3,3) rectangle (.8,-3);
         \draw[fill=white] (-1.9,1) rectangle (2.4,-1);
         \node at (.25,0) {$\JW_R(\prescript{}{\tblu}{n})$};
         \node at (.3,1.4) {$\dots$}; 
         \node at (.3,-1.4) {$\dots$};
         \end{scope}
         \draw[dashed] (-2.3,2) to (2.8,2);
         \draw[dashed] (-2.3,-2) to (2.8,-2);
       \end{tikzpicture}
\end{gathered}
=
\begin{gathered}
  \begin{tikzpicture}[xscale=.45,yscale=-.35]
    \begin{scope}
         \clip (-3,2) rectangle (3.5,-2);
    \draw[fill=dgrmblu] (-3.3,3) rectangle (1.4,-3);
    \draw[fill=dgrmred] (3.8,3) rectangle (1.4,-3);
             \draw[fill=dgrmred] (-.9,3) rectangle (0.5,-3);
             \draw[fill=white] (-.3,3) rectangle (0.8,-3);
             \draw[fill=dgrmblu] (2,3) to (2,-1)
         to[out=-90,in=180] (2.5,-1.5) to[out=0,in=-90] (3,-1) to
         (3,3) to (2,3);
         \draw[fill=dgrmred] (-1.5,-3) to (-1.5,1)
         to[out=90,in=0] (-2,1.5) to[out=180,in=90] (-2.5,1) to
         (-2.5,-3) to (-1.5,-3);
         \draw[fill=white] (-1.9,1) rectangle (2.4,-1);
         \node at (.25,0) {$\JW_R(\prescript{}{\sred}{n})$};
         \node at (.3,1.4) {$\dots$}; 
         \node at (.3,-1.4) {$\dots$};
         \end{scope}
         \draw[dashed] (-3,2) to (3.5,2);
         \draw[dashed] (-3,-2) to (3.5,-2);
       \end{tikzpicture}
\end{gathered} & \text{($n$ odd)} \\
\begin{gathered}
  \begin{tikzpicture}[xscale=.45,yscale=.35]
    \begin{scope}
         \clip (-3,2) rectangle (3.5,-2);
    \draw[fill=dgrmblu] (-3.3,3) rectangle (3.8,-3);
             \draw[fill=dgrmred] (-.9,3) rectangle (1.4,-3);
             \draw[fill=white] (-.3,3) rectangle (.8,-3);
             \draw[fill=dgrmred] (2,3) to (2,-1)
         to[out=-90,in=180] (2.5,-1.5) to[out=0,in=-90] (3,-1) to
         (3,3) to (2,3);
         \draw[fill=dgrmred] (-1.5,-3) to (-1.5,1)
         to[out=90,in=0] (-2,1.5) to[out=180,in=90] (-2.5,1) to
         (-2.5,-3) to (-1.5,-3);
         \draw[fill=white] (-1.9,1) rectangle (2.4,-1);
         \node at (.25,0) {$\JW_R(\prescript{}{\sred}{n})$};
         \node at (.3,1.4) {$\dots$}; 
         \node at (.3,-1.4) {$\dots$};
         \end{scope}
         \draw[dashed] (-3,2) to (3.5,2);
         \draw[dashed] (-3,-2) to (3.5,-2);
       \end{tikzpicture}
\end{gathered}
    & =
     \lambda
\begin{gathered}
  \begin{tikzpicture}[xscale=.45,yscale=.35]
    \begin{scope}
         \clip (-2.3,2) rectangle (2.8,-2);
    \draw[fill=dgrmblu] (-3.3,3) rectangle (3.8,-3);
             \draw[fill=dgrmred] (-1.5,3) rectangle (2,-3);
             \draw[fill=dgrmblu] (-.9,3) rectangle (1.4,-3);
             \draw[fill=white] (-.3,3) rectangle (.8,-3);
         \draw[fill=white] (-1.9,1) rectangle (2.4,-1);
         \node at (.25,0) {$\JW_R(\prescript{}{\tblu}{n})$};
         \node at (.3,1.4) {$\dots$}; 
         \node at (.3,-1.4) {$\dots$};
         \end{scope}
         \draw[dashed] (-2.3,2) to (2.8,2);
         \draw[dashed] (-2.3,-2) to (2.8,-2);
       \end{tikzpicture}
\end{gathered}
=
\begin{gathered}
  \begin{tikzpicture}[xscale=.45,yscale=-0.35]
    \begin{scope}
         \clip (-3,2) rectangle (3.5,-2);
    \draw[fill=dgrmblu] (-3.3,3) rectangle (3.8,-3);
             \draw[fill=dgrmred] (-.9,3) rectangle (1.4,-3);
             \draw[fill=white] (-.3,3) rectangle (.8,-3);
             \draw[fill=dgrmred] (2,3) to (2,-1)
         to[out=-90,in=180] (2.5,-1.5) to[out=0,in=-90] (3,-1) to
         (3,3) to (2,3);
         \draw[fill=dgrmred] (-1.5,-3) to (-1.5,1)
         to[out=90,in=0] (-2,1.5) to[out=180,in=90] (-2.5,1) to
         (-2.5,-3) to (-1.5,-3);
         \draw[fill=white] (-1.9,1) rectangle (2.4,-1);
         \node at (.25,0) {$\JW_R(\prescript{}{\sred}{n})$};
         \node at (.3,1.4) {$\dots$}; 
         \node at (.3,-1.4) {$\dots$};
         \end{scope}
         \draw[dashed] (-3,2) to (3.5,2);
         \draw[dashed] (-3,-2) to (3.5,-2);
       \end{tikzpicture}
\end{gathered} & \text{($n$ even)}
\end{align*}
Our second main result gives a combined condition for the existence and rotatability of two-colored Jones--Wenzl projectors.

\begin{thmletter} \label{existsrotates}
The two-colored Jones--Wenzl projectors $\JW_R(\prescript{}{\sred}{n})$ and $\JW_R(\prescript{}{\tblu}{n})$ exist and are rotatable if and only if $\qbinom{n+1}{k}_{\sred}=\qbinom{n+1}{k}_{\tblu}=0$ in $R$ for all integers $1 \leq k \leq n$.
\end{thmletter}

The same algebraic condition was first introduced by Abe in the context of the Hecke category \cite[Assumption~1.1]{abe-homBS}, which we discuss below.


\subsection*{The Hecke category}

The two-colored Temperley--Lieb algebra lies at the heart of the Elias--Williamson diagrammatic Hecke category \cite{ew-soergelcalc}. 
In more detail, the diagrammatic Hecke category is only well defined when certain two-colored Jones--Wenzl projectors exist and are rotatable. 
Elias--Williamson initially gave an incorrect algebraic condition for rotatability in \cite[{(3.3)}]{ew-soergelcalc}; they later identified and partially corrected this error in \cite[\S 5]{ew-localizedcalc}. 
Our rotatability condition in \cref{existsrotates} is enough to completely correct this error. 

\begin{corletter} \label{Hwelldefined}
In the absence of parabolic type $H_3$ subgroups (see \cref{H3defined}), the diagrammatic Hecke category is well defined if and only if the underlying realization is an Abe realization (see \cite[Assumption~1.1]{abe-homBS} or \cref{realizcorrected}). 
\end{corletter}

Recently Abe has shown that there is a ``bimodule-theoretic'' category (a modification of the category of classical Soergel bimodules) which under mild conditions is equivalent to the diagrammatic Hecke category when the underlying realization is an Abe realization \cite{abe-bimodhecke,abe-homBS}.
An important consequence of \cref{Hwelldefined} is that this equivalence essentially always holds. 

\begin{corletter} \label{abeequalsew}
Assume Demazure surjectivity holds (see \cref{heckecategorification}), and that the base ring is a Noetherian domain. 
If the diagrammatic Hecke category is well defined, it is equivalent to Abe's category of Soergel bimodules.
\end{corletter}

We find it noteworthy that our result gives the best possible equivalence result for two seemingly distinct categorifications of the Hecke algebra.

\subsection*{Acknowledgments}

We thank the anonymous referee for their helpful comments and suggestions.
We are also grateful for financial support from the Royal Commission for the Exhibition of 1851 and EPSRC (EP/V00090X/1).

\section{Preliminaries}
\label{prelim}

Let $A=\ZZ[x_{\sred},x_{\tblu}]$ be the integral polynomial ring in two variables. 
The \defnemph{two-colored quantum numbers} are defined as follows. 
First set $[1]_{\sred}=[1]_{\tblu}=1$, $[2]_{\sred}=x_{\sred}$, and $[2]_{\tblu}=x_{\tblu}$ in $A$. 
For $n>1$ we inductively define
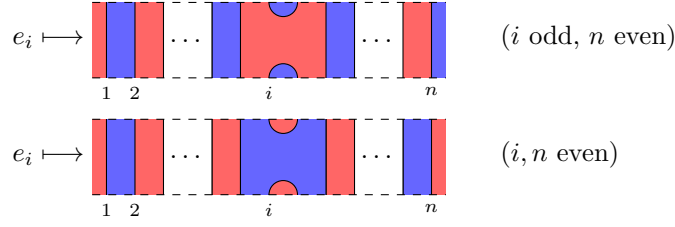
\begin{align}
[n+1]_{\sred}& =[2]_{\sred} [n]_{\tblu} - [n-1]_{\sred} \text{,} & [n+1]_{\tblu} & =[2]_{\tblu} [n]_{\sred} - [n-1]_{\tblu} \text{.} \label{eq:twocolqnum}
\end{align}
These formulas can be rearranged to inductively define $[n]_{\sred}$ and $[n]_{\tblu}$ for $n \leq 0$. 
For a commutative $A$-algebra $R$, we also define two-colored quantum numbers in $R$ to be the specializations of two-colored quantum numbers in $A$, which we will write in the same way.

These polynomials are bivariate extensions of the usual (one-colored) quantum numbers, which can be recovered as follows.
Let $\overline{A}=A/(x_{\sred}-x_{\tblu}) \iso \ZZ[x]$, where $x$ is the image of $x_{\sred}$ or $x_{\tblu}$.
Then the one-colored quantum number $[n]$ is the image of $[n]_{\sred}$ or $[n]_{\tblu}$ in $\overline{A}$.
When $n$ is odd, $[n]$ is an even polynomial, so we can formally evaluate $[n]$ at $x=\sqrt{x_{\sred} x_{\tblu}}$ to obtain an element of $A$.
When $n$ is even, $[n]/[2]$ is an even polynomial, which we can similarly formally evaluate at $x=\sqrt{x_{\sred} x_{\tblu}}$.
In both cases, it is easy to show by induction that
\begin{equation} \label{eq:colordependence}
\begin{aligned}
[n]_{\sred}& =[n](\sqrt{x_{\sred} x_{\tblu}})=[n]_{\tblu} & & \text{if $n$ is odd,} \\
\frac{[n]_{\sred}}{[2]_{\sred}}& =\left(\frac{[n]}{[2]}\right)(\sqrt{x_{\sred} x_{\tblu}})=\frac{[n]_{\tblu}}{[2]_{\tblu}} & & \text{if $n$ is even}
\end{aligned}
\end{equation}
in $A$.
In other words, two-colored quantum numbers are essentially the same as ordinary quantum numbers up to a factor of $[2]_{\sred}$ and $[2]_{\tblu}$ depending on color.

It is self-evident that the automorphism of $A$ which exchanges $x_{\sred}$ and $x_{\tblu}$ (``color swap'') also exchanges $[n]_{\sred}$ and $[n]_{\tblu}$ for all $n$.
For this reason, we will generally write statements only for $[n]_{\sred}$ and leave it to the reader to formulate color-swapped analogues.
Similarly we have $2\TL(\prescript{}{\sred}{n}; [2]_{\sred},[2]_{\tblu}) \iso 2\TL(\prescript{}{\tblu}{n}; [2]_{\tblu},[2]_{\sred})$, and this isomorphism maps $\JW_R(\prescript{}{\sred}{n})$ to $\JW_R(\prescript{}{\tblu}{n})$ when they exist, so we will only state our results for $2\TL_R(\prescript{}{\sred}{n})$ and $\JW_R(\prescript{}{\sred}{n})$.

Let $D=e_{i_1} e_{i_2} \dotsb e_{i_r}$ be a monomial of length $r$ in the generators of $2\TL(\prescript{}{\sred}{n})$.
We say that $D$ is \defnemph{reduced} if it cannot be rewritten as a monomial $e_{j_1} e_{j_2} \dotsm e_{j_s}$ in the generators using \eqref{eq:oddquadratic}--\eqref{eq:braidreln} for some $s<r$.
As mentioned in \cref{intro}, the two-colored Temperley--Lieb algebra $2\TL(\prescript{}{\sred}{n})$ has a basis consisting of these reduced monomials. 
As in the one-colored case, there is a bijection between this basis in $2\TL(\prescript{}{\sred}{n})$ and (isotopy classes of) two-colored Temperley--Lieb diagrams whose leftmost region is colored red which induces an isomorphism between the algebraic and diagrammatic versions of $2\TL(\prescript{}{\sred}{n})$. 
(For a careful proof of this fact in the one-colored case see e.g.~\cite[Theorem~2.4]{ridout-st-aubin}.)
Under this isomorphism we have
\begingroup
\allowdisplaybreaks
\begin{align*}
e_i & \longmapsto
  \begin{tikzpicture}[xscale=.75,yscale=.25,baseline=(base)]
    \path (0,-6pt) coordinate (base);
    \begin{scope}
        \clip (-3.5,2) rectangle (2.7,-2);
        \draw[fill=dgrmred] (-3.8-0.25,3) rectangle (2.2+0.25,-3);
        \draw[fill=dgrmblu] (2.7+0.25,3) rectangle (2.2+0.25,-3);
        \draw[fill=dgrmblu] (-3-0.25,3) rectangle (-2.5-0.25,-3);
        \draw[fill=dgrmblu] (-1.4,3) rectangle (-0.9,-3);
        \draw[fill=dgrmblu] (0.6,3) rectangle (1.1,-3);
        \draw[fill=white] (-2-0.25,3) rectangle (-1.4,-3);
        \draw[fill=white] (1.1,3) rectangle (1.7+0.25,-3);
        \draw[fill=dgrmblu] (-0.4,3) to (-0.4,2) arc[x radius=0.25, y radius=0.75, start angle=180, end angle=360]
        to (0.1,3);
        \draw[fill=dgrmblu] (-0.4,-3) to (-0.4,-2) arc[x radius=0.25, y radius=0.75, start angle=180, end angle=0]
        to (0.1,-3);
        \node at (-2.25*0.5-1.4*0.5,0) {$\dots$}; 
        \node at (1.1*0.5+1.95*0.5,0) {$\dots$};
     \end{scope}
         \draw[dashed] (-3.5,2) to (2.7,2);
         \draw[dashed] (-3.5,-2) to (2.7,-2);
         \path (-3-0.25,-2) node[below, font=\scriptsize] {$1$};
         \path (-2.5-0.25,-2) node[below, font=\scriptsize] {$2$};
         \path (2.2+0.25,-2) node[below, font=\scriptsize] {$n$};
         \path (-0.4,-2) node[below, font=\scriptsize] {$i$};
  \end{tikzpicture}
   \qquad \text{($i,n$ odd)} \\
e_i & \longmapsto
  \begin{tikzpicture}[xscale=.75,yscale=.25,baseline=(base)]
    \path (0,-6pt) coordinate (base);
    \begin{scope}
        \clip (-3.5,2) rectangle (2.7,-2);
        \draw[fill=dgrmred] (-3.8-0.25,3) rectangle (2.2+0.25,-3);
        \draw[fill=dgrmblu] (2.7+0.25,3) rectangle (2.2+0.25,-3);
        \draw[fill=dgrmblu] (-3-0.25,3) rectangle (-2.5-0.25,-3);
        \draw[fill=dgrmred] (-1.4,3) rectangle (-0.9,-3);
        \draw[fill=dgrmred] (0.6,3) rectangle (1.1,-3);
        \draw[fill=dgrmblu] (-0.9,-3) rectangle (0.6,3);
        \draw[fill=white] (-2-0.25,3) rectangle (-1.4,-3);
        \draw[fill=white] (1.1,3) rectangle (1.7+0.25,-3);
        \draw[fill=dgrmred] (-0.4,3) to (-0.4,2) arc[x radius=0.25, y radius=0.75, start angle=180, end angle=360]
        to (0.1,3);
        \draw[fill=dgrmred] (-0.4,-3) to (-0.4,-2) arc[x radius=0.25, y radius=0.75, start angle=180, end angle=0]
        to (0.1,-3);
        \node at (-2.25*0.5-1.4*0.5,0) {$\dots$}; 
        \node at (1.1*0.5+1.95*0.5,0) {$\dots$};
     \end{scope}
         \draw[dashed] (-3.5,2) to (2.7,2);
         \draw[dashed] (-3.5,-2) to (2.7,-2);
         \path (-3-0.25,-2) node[below, font=\scriptsize] {$1$};
         \path (-2.5-0.25,-2) node[below, font=\scriptsize] {$2$};
         \path (2.2+0.25,-2) node[below, font=\scriptsize] {$n$};
         \path (-0.4,-2) node[below, font=\scriptsize] {$i$};
  \end{tikzpicture}
 \qquad \text{($i$ even, $n$ odd)} \\
e_i & \longmapsto 
  \begin{tikzpicture}[xscale=.75,yscale=.25,baseline=(base)]
    \path (0,-6pt) coordinate (base);
    \begin{scope}
        \clip (-3.5,2) rectangle (2.7,-2);
        \draw[fill=dgrmred] (-3.8-0.25,3) rectangle (2.2+0.25,-3);
        \draw[fill=dgrmblu] (2.7+0.25,3) rectangle (2.2+0.25,-3);
        \draw[fill=dgrmblu] (-3-0.25,3) rectangle (-2.5-0.25,-3);
        \draw[fill=dgrmblu] (-1.4,3) rectangle (-0.9,-3);
        \draw[fill=dgrmblu] (0.6,3) rectangle (1.1,-3);
        \draw[fill=white] (-2-0.25,3) rectangle (-1.4,-3);
        \draw[fill=white] (1.1,3) rectangle (1.7+0.25,-3);
        \draw[fill=dgrmblu] (-0.4,3) to (-0.4,2) arc[x radius=0.25, y radius=0.75, start angle=180, end angle=360]
        to (0.1,3);
        \draw[fill=dgrmblu] (-0.4,-3) to (-0.4,-2) arc[x radius=0.25, y radius=0.75, start angle=180, end angle=0]
        to (0.1,-3);
        \node at (-2.25*0.5-1.4*0.5,0) {$\dots$}; 
        \node at (1.1*0.5+1.95*0.5,0) {$\dots$};
     \end{scope}
         \draw[dashed] (-3.5,2) to (2.7,2);
         \draw[dashed] (-3.5,-2) to (2.7,-2);
         \path (-3-0.25,-2) node[below, font=\scriptsize] {$1$};
         \path (-2.5-0.25,-2) node[below, font=\scriptsize] {$2$};
         \path (2.2+0.25,-2) node[below, font=\scriptsize] {$n$};
         \path (-0.4,-2) node[below, font=\scriptsize] {$i$};
  \end{tikzpicture}
 \qquad \text{($i$ odd, $n$ even)} \\
e_i & \longmapsto 
  \begin{tikzpicture}[xscale=.75,yscale=.25,baseline=(base)]
    \path (0,-6pt) coordinate (base);
    \begin{scope}
        \clip (-3.5,2) rectangle (2.7,-2);
        \draw[fill=dgrmred] (-3.8-0.25,3) rectangle (3,-3);
        \draw[fill=dgrmblu] (2.2+0.25,3) rectangle (1.7+0.25,-3);
        \draw[fill=dgrmblu] (-3-0.25,3) rectangle (-2.5-0.25,-3);
        \draw[fill=dgrmred] (-1.4,3) rectangle (-0.9,-3);
        \draw[fill=dgrmred] (0.6,3) rectangle (1.1,-3);
        \draw[fill=dgrmblu] (-0.9,-3) rectangle (0.6,3);
        \draw[fill=white] (-2-0.25,3) rectangle (-1.4,-3);
        \draw[fill=white] (1.1,3) rectangle (1.7+0.25,-3);
        \draw[fill=dgrmred] (-0.4,3) to (-0.4,2) arc[x radius=0.25, y radius=0.75, start angle=180, end angle=360]
        to (0.1,3);
        \draw[fill=dgrmred] (-0.4,-3) to (-0.4,-2) arc[x radius=0.25, y radius=0.75, start angle=180, end angle=0]
        to (0.1,-3);
        \node at (-2.25*0.5-1.4*0.5,0) {$\dots$}; 
        \node at (1.1*0.5+1.95*0.5,0) {$\dots$};
     \end{scope}
         \draw[dashed] (-3.5,2) to (2.7,2);
         \draw[dashed] (-3.5,-2) to (2.7,-2);
         \path (-3-0.25,-2) node[below, font=\scriptsize] {$1$};
         \path (-2.5-0.25,-2) node[below, font=\scriptsize] {$2$};
         \path (2.2+0.25,-2) node[below, font=\scriptsize] {$n$};
         \path (-0.4,-2) node[below, font=\scriptsize] {$i$};
  \end{tikzpicture} 
  \qquad  \text{($i,n$ even)} 
\end{align*}
\endgroup
Given an element $f \in 2\TL_R(\prescript{}{\sred}{n})$ and a two-colored Temperley--Lieb diagram $D$ we will write
\begin{equation*}
\coeff_{{} \in f} D
\end{equation*}
for the coefficient of $D$ when $f$ is written in the diagrammatic basis.

If $R$ is a commutative $A$-algebra for which $\JW_R(\prescript{}{\sred}{n})$ exists for all $n$, then the coefficients of $\JW_{R}(\prescript{}{\sred}{n})$ can be calculated inductively as follows.
Suppose $D$ is a two-colored Temperley--Lieb diagram in $2\TL_R(\prescript{}{\sred}{(n+1)})$.
Let $\hat{D}$ be the diagram with $n+2$ bottom boundary points and $n$ top boundary points obtained by folding down the strand connected to the top right boundary point of $D$.
If there is a strand connecting the $i$th and $(i+1)$th bottom boundary points of $\hat{D}$, let $D_i$ denote the two-colored Temperley--Lieb diagram with $n$ strands so obtained by deleting this cap. 
For example, if 
\begin{equation*}
D=\begin{gathered}
        \begin{tikzpicture}[xscale=.4,yscale=.2]
          \begin{scope}
            \clip (-3,3.5) rectangle (3,-3.5);
            \draw[fill=dgrmred] (2,4) rectangle (-3,-4);
            \draw[fill=dgrmblu] (2,4) rectangle (3,-4);
          \draw[fill=dgrmblu] (-2,-4) to (-2,4) to (-1,3.5) to[out=-90,in=90] (1,-3.5) to (1,-4) to (0,-3.5) to[out=90,in=0] (-.5,-2.5) to[out=180,in=90] (-1,-3.5) to (-2,-4);
          \draw[fill=dgrmblu] (0,4) to (0,3.5) to[out=-90,in=180] (.5,2.5) to[out=0,in=-90] (1,3.5) to (0,4);          
       \end{scope}
       \draw[dashed] (-3,3.5) to (3,3.5);
       \draw[dashed] (-3,-3.5) to (3,-3.5);
     \end{tikzpicture}
\end{gathered}
\end{equation*}
then
\begin{equation*}
\hat{D}=\begin{gathered}
        \begin{tikzpicture}[xscale=.4,yscale=.2]
          \begin{scope}
            \clip (-3,3.5) rectangle (3.5,-3.5);
            \draw[fill=dgrmred] (3.5,4) rectangle (-3,-4);
            \draw[fill=dgrmblu] (2,-4) to (2,-3.5) to[out=90,in=180] (2.5,-2.5) to[out=0,in=90] (3,-3.5) to (2,-4);
          \draw[fill=dgrmblu] (-2,-4) to (-2,4) to (-1,3.5) to[out=-90,in=90] (1,-3.5) to (1,-4) to (0,-3.5) to[out=90,in=0] (-.5,-2.5) to[out=180,in=90] (-1,-3.5) to (-2,-4);
          \draw[fill=dgrmblu] (0,4) to (0,3.5) to[out=-90,in=180] (.5,2.5) to[out=0,in=-90] (1,3.5) to (0,4);          
       \end{scope}
       \draw[dashed] (-3,3.5) to (3.5,3.5);
       \draw[dashed] (-3,-3.5) to (3.5,-3.5);
     \end{tikzpicture} \\
\end{gathered}
\end{equation*}
and
\begin{equation*}
D_2=\begin{gathered}
        \begin{tikzpicture}[xscale=.4,yscale=.2]
          \begin{scope}
            \clip (-3,3.5) rectangle (2,-3.5);
            \draw[fill=dgrmred] (3.5,4) rectangle (-3,-4);
            \draw[fill=dgrmblu] (-2,4) rectangle (-1,-4);
          \draw[fill=dgrmblu] (0,-4) to (0,-3.5) to[out=90,in=-180] (.5,-2.5) to[out=0,in=90] (1,-3.5) to (0,-4);          
          \draw[fill=dgrmblu] (0,4) to (0,3.5) to[out=-90,in=180] (.5,2.5) to[out=0,in=-90] (1,3.5) to (0,4);          
       \end{scope}
       \draw[dashed] (-3,3.5) to (2,3.5);
       \draw[dashed] (-3,-3.5) to (2,-3.5);
     \end{tikzpicture} \\
\end{gathered} \text{,} \qquad \qquad \qquad
D_5=\begin{gathered}
        \begin{tikzpicture}[xscale=.4,yscale=.2]
          \begin{scope}
            \clip (-3,3.5) rectangle (2,-3.5);
            \draw[fill=dgrmred] (3.5,4) rectangle (-3,-4);
          \draw[fill=dgrmblu] (-2,-4) to (-2,4) to (-1,3.5) to[out=-90,in=90] (1,-3.5) to (1,-4) to (0,-3.5) to[out=90,in=0] (-.5,-2.5) to[out=180,in=90] (-1,-3.5) to (-2,-4);
          \draw[fill=dgrmblu] (0,4) to (0,3.5) to[out=-90,in=180] (.5,2.5) to[out=0,in=-90] (1,3.5) to (0,4);          
       \end{scope}
       \draw[dashed] (-3,3.5) to (2,3.5);
       \draw[dashed] (-3,-3.5) to (2,-3.5);
     \end{tikzpicture} \\
\end{gathered} \text{.}
\end{equation*}

\begin{thm} \label{JWcoef}
Suppose $\JW_R(\prescript{}{\sred}{n})$ and $\JW_R(\prescript{}{\sred}{(n+1)})$ both exist.
Then $[n+1]_{\sred}$ is invertible, and we have
\begin{equation*}
\coeff_{{} \in \JW_{R}(\prescript{}{\sred}{(n+1)})} D=\sum_{i} \frac{[i]_{\upur}}{[n+1]_{\sred}}\coeff_{{} \in \JW_{R}(\prescript{}{\sred}{n})} D_i \text{,}
\end{equation*}
where the sum is taken over all positions $i$ where $D_i$ is defined, and $\upur$ is the color of the deleted cap.
\end{thm}

\begin{proof}
The argument in the one-color setting (see \cite[Proposition~4.1]{morrison-JW} or \cite[Corollary~3.7]{frenkelkhovanov}) follows essentially unchanged from \cite[(6.29)]{ew-localizedcalc}.
\end{proof}

By a similar computation it can be shown that $\JW_{\Frac A}(\prescript{}{\sred}{n})$ exists for all $n \in \NN$ (see e.g.~\cite[Theorem~6.14]{ew-localizedcalc}).
We will carefully show later that this computation is ``generic'', i.e.~if $\JW_R(\prescript{}{\sred}{n})$ exists, then its coefficients are specializations of the coefficients of $\JW_{\Frac A}(n)$.

The existence criterion in \cref{existence} is known to hold in the one-color setting, i.e.~when the images of $x_{\sred}$ and $x_{\tblu}$ in $R$ are equal. 
In these circumstances we write $\TL_R(n)$ and $\JW_R(n)$ for the one-color Temperley--Lieb algebra and Jones--Wenzl projector.

\begin{thm}[{\cite[Theorem~A.2]{el-univcox}}] \label{existence-onecol}
Suppose $R$ is a commutative $A$-algebra which factors through $\overline{A}$. 
Then $\JW_R(n)$ exists if and only if the one-color quantum binomial coefficients
\begin{equation*}
\qbinom{n}{k}=\frac{[n]!}{[k]![n-k]!}=\frac{[n][n-1]\dotsm [n-k+1]}{[k][k-1] \dotsm [1]}
\end{equation*}
are invertible in $R$ for all integers $0 \leq k \leq n$.
\end{thm}

In light of the ``generic'' nature of the coefficients of $\JW_R(\prescript{}{\sred}{n})$, we can interpret \cref{existence-onecol} as description of the denominators of the coefficients of $\JW_{\Frac \overline{A}}(n)$.
Unfortunately, none of the known proofs of this result (most of which use connections to Lie theory in a crucial way) generalize easily to the two-colored setting.

Finally, we will give an alternative criterion for checking rotatability.
For $f \in 2\TL_R(\prescript{}{\sred}{n})$ define the \defnemph{partial trace} of $f$ to be
\begin{align*}
\pTr(f)& =\begin{gathered}
\begin{tikzpicture}[xscale=.6,yscale=0.8]
        \begin{scope}
          \clip (-1.5,1) rectangle (3,-1);
  \draw[fill=dgrmblu] (1,2) rectangle (3.5,-2);
  \draw[fill=dgrmred] (1.5,.6) to[out=90,in=180] (2,.8)
  to[out=0,in=90] (2.5,0) to[out=-90,in=0] (2,-.8) to[out=180,in=-90]
  (1.5,-.6) to (1.5,.6);
  \draw[fill=dgrmred] (0,2) rectangle (-2,-2);
  \draw[fill=dgrmblu] (-1,2) rectangle (-.5,-2);
  \draw[fill=white] (-1.25,.6) rectangle (1.9,-.6);
  \node at (.25,0) {$f$};
  \node at (.55,.8) {$\dots$};
  \node at (.55,-.8) {$\dots$};
  \end{scope}
  \draw[dashed] (-1.5,1) to (3,1);
  \draw[dashed] (-1.5,-1) to (3,-1);
\end{tikzpicture}
\end{gathered} & & 
\text{($n$ odd)} \\
\pTr(f)& =\begin{gathered}
\begin{tikzpicture}[xscale=.6,yscale=0.8]
        \begin{scope}
          \clip (-1.5,1) rectangle (3,-1);
  \draw[fill=dgrmred] (1,2) rectangle (3.5,-2);
  \draw[fill=dgrmblu] (1.5,.6) to[out=90,in=180] (2,.8)
  to[out=0,in=90] (2.5,0) to[out=-90,in=0] (2,-.8) to[out=180,in=-90]
  (1.5,-.6) to (1.5,.6);
  \draw[fill=dgrmred] (0,2) rectangle (-2,-2);
  \draw[fill=dgrmblu] (-1,2) rectangle (-.5,-2);
  \draw[fill=white] (-1.25,.6) rectangle (1.9,-.6);
  \node at (.25,0) {$f$};
  \node at (.55,.8) {$\dots$};
  \node at (.55,-.8) {$\dots$};
  \end{scope}
  \draw[dashed] (-1.5,1) to (3,1);
  \draw[dashed] (-1.5,-1) to (3,-1);
\end{tikzpicture}
\end{gathered} & &
\text{($n$ even)}
\end{align*}
From the definition of the Jones--Wenzl projector, it is easy to see that $\JW_R(\prescript{}{\sred}{n})$ is rotatable if and only if $\pTr(\JW_R(\prescript{}{\sred}{n}))=0$.
Using entirely standard techniques (e.g.~\cite[\S 6.6]{ew-localizedcalc}), one can show that
\begin{equation}
\pTr(\JW_R(\prescript{}{\sred}{n}))=-\frac{[n+1]_{\sred}}{[n]_{\sred}}\JW_R(\prescript{}{\sred}{(n-1)}) \label{eq:genericpTr}
\end{equation}
when both $\JW_R(\prescript{}{\sred}{n})$ and $\JW_R(\prescript{}{\sred}{(n-1)})$ exist. 
This gives the following partial rotatability criterion.

\begin{prop} \label{genericrotatability}
Suppose both $\JW_R(\prescript{}{\sred}{n})$ and $\JW_R(\prescript{}{\sred}{(n-1)})$ exist. 
Then $\JW_R(\prescript{}{\sred}{n})$ is rotatable if and only if $[n+1]_{\sred}=0$.
\end{prop}

The key to proving the full rotatability criterion will be to interpret \eqref{eq:genericpTr} generically.

\section{Principal ideals}

In this section, we show that several ideals generated by certain two-colored quantum numbers and binomial coefficients are principal.
Recall that for ordinary quantum numbers, one can show that if $d|n$ then $[d]|[n]$. 
Using $\eqref{eq:colordependence}$ it immediately follows that $[d]_{\sred}|[n]_{\sred}$.

\begin{lem}[Quantum B\'{e}zout's identity] \label{qbezout}
Let $m,n \in \NN$. 
There exist polynomials $a,b \in A$ such that 
\begin{equation*}
a[m]_{\sred}+b[n]_{\sred}=[\gcd(m,n)]_{\sred} \text{.}
\end{equation*}
\end{lem}

\begin{proof}
Suppose without loss of generality that $m<n$. 
We will show that the ideal in $A$ generated by $[m]_{\sred}$ and $[n]_{\sred}$ contains $[n-m]_{\sred}$.
If $m$ and $n$ are not both odd, then
\begin{align*}
\begin{split}
[n-1]_{\tblu}[m]_{\sred} - [m-1]_{\tblu}[n]_{\sred} & =
([m+n-2]_{\sred}+[m+n-4]_{\sred}+\dotsb+[-(n-m)+2]_{\sred}) \\
& \quad -([m+n-2]_{\sred}+[m+n-4]_{\sred}+\dotsb+[n-m+2]_{\sred})
\end{split} \\
& =[n-m]_{\sred}+[n-m-2]_{\sred}+\dotsb+[-(n-m)+2]_{\sred} \\
& =[n-m]_{\sred}
\end{align*}
by \cite[(6.5a)--(6.5c)]{ew-localizedcalc}. 
If $m$ and $n$ are both odd, a similar calculation yields 
\begin{equation*}
[n-1]_{\sred}[m]_{\sred} - [m-1]_{\sred}[n]_{\sred}=[n-m]_{\sred} \text{.}
\end{equation*}
By repeating this step multiple times, we can run Euclid's algorithm, and the result follows.
\end{proof}

Next we introduce the cyclotomic parts of quantum numbers, which are roughly analogous to cyclotomic polynomials.
Recall that the one-color quantum numbers are renormalizations of Chebyshev polynomials of the second kind. 
More precisely, if we evaluate a quantum number $[n]$ at $x=2\cos \theta$, we obtain
\begin{equation*}
[n](2\cos \theta) = \frac{\sin n\theta}{\sin \theta} \text{.}
\end{equation*}
Since $[n]$ is a monic polynomial in $x$ of degree $n-1$ we conclude that 
\begin{equation*}
[n]=\prod_{k=1}^{n-1} \left(x-2\cos \frac{k\pi}{n}\right)
\end{equation*}
We define the \defnemph{cyclotomic part} of the one-color quantum number $[n]$ to be the polynomial
\begin{equation*}
\Theta_n = \prod_{\substack{1 \leq k < n\\ (k,n)=1}} \left(x-2\cos \frac{k\pi}{n}\right) \text{.}
\end{equation*}

\begin{lem} \label{cyclofacts}
Let $n \in \NN$. 
We have
\begin{enumerate}[label={\rm (\roman*)}]
\item \label{item:degree} $\Theta_n \in \ZZ[x]$, and $\deg \Theta_n=\varphi(n)$ when $n>1$;

\item \label{item:prod} $[n]=\prod_{k|n} \Theta_n$;

\item \label{item:mobinv} $\Theta_n=\prod_{k|n} [k]^{\mu(n/k)}$, where $\mu:\NN \rightarrow \{\pm 1\}$ is the M\"{o}bius function.
\end{enumerate}
Moreover, if $n>2$ then we also have $\Theta_n(x)=\Psi_n(x^2)$, where $\Psi_n \in \ZZ[x]$ is the minimal polynomial of $4\cos^2(\pi/n)$.
\end{lem}

\begin{proof}
Both \ref{item:degree} and \ref{item:prod} follow from the definition and basic properties of cyclotomic fields and algebraic integers.
Applying M\"obius inversion to \ref{item:prod} yields \ref{item:mobinv}. 
For the final claim, we observe that if $n>2$ then $\Theta_n$ is an even polynomial, so is of the form of $\Psi_n(x^2)$ for some $\Psi_n \in \ZZ[x]$ of degree $\varphi(n)/2$. 
By construction $4\cos^2(\pi/n)$ is a root of $\Psi_n$. 
Since
\begin{equation*}
4\cos^2 \frac{\pi}{n}=2\cos \frac{2\pi}{n}+2
\end{equation*}
and $\QQ(2\cos(2\pi/n)+2)=\QQ(\cos(2\pi/n))$ is a field extension of $\QQ$ of degree $\varphi(n)/2$, $\Psi_n$ must be the minimal polynomial of $4\cos^2(\pi/n)$.
\end{proof}

\begin{defn}
For $n \in \NN$, we define the \defnemph{cyclotomic part} of the two-colored quantum number $[n]_{\sred}$ to be
\begin{equation*}
\Theta_{n,\sred}=\begin{cases}
\Psi_n(x_{\sred} x_{\tblu}) & \text{if $n>2$,} \\
x_{\sred} & \text{if $n=2$,} \\
1 & \text{if $n=1$.}
\end{cases}
\end{equation*}
\end{defn}

Using \eqref{eq:colordependence} and \cref{cyclofacts} we similarly obtain $[n]_{\sred}=\prod_{k|n} \Theta_{n,\sred}$ and $\Theta_{n,\sred}=\prod_{k|n} [n]_{\sred}^{\mu(n/k)}$.

\begin{lem}
The polynomials $\Theta_{n,\sred}$ are all irreducible and distinct in $A$ (but note that $\Theta_{n,\sred}=\Theta_{n,\tblu}$ if $n>2$).
\end{lem}

\begin{proof}
Irreducibility is clear when $n=2$.
When $n>2$, we have $\Theta_{n,\sred}=\Theta_{n,\tblu}=\Psi_n(x_{\sred} x_{\tblu})$, which is irreducible because $\Psi_n$ is (see e.g.~\cite[(3.3)]{rotthaus}).
Distinctness follows as well because the polynomials $\Psi_n$ are distinct.
\end{proof}

\begin{lem} \label{thetabezout}
Let $m,n \in \NN$ such that $m \nmid n$ and $n \nmid m$.
There exist polynomials $a,b \in A$ such that 
\begin{equation*}
a\Theta_{m,\sred}+b\Theta_{n,\sred}=1 \text{.}
\end{equation*}
\end{lem}

\begin{proof}
Suppose without loss of generality that $m<n$, and let $d=\gcd(m,n)$.
By \cref{qbezout} there exist $a',b' \in A$ such that
\begin{equation*}
a'[m]_{\sred}+b'[n]_{\sred}=[d]_{\sred} \text{.}
\end{equation*}
By assumption $d<m<n$, so we have
\begin{align*}
\frac{[m]_{\sred}}{[d]_{\sred}} & \in \Theta_{m,\sred} A \\
\frac{[n]_{\sred}}{[d]_{\sred}} & \in \Theta_{n,\sred} A
\end{align*}
and thus dividing by $[d]_{\sred}$ we obtain
\begin{equation*}
a\Theta_{m,\sred}+b\Theta_{n,\sred}=1 \text{.} \qedhere
\end{equation*}
\end{proof}

\begin{prop} \label{thetaprincipal}
Let $m_1,m_2,\dotsc,m_k,n_1,n_2,\dotsc,n_l \in \NN$ such that for all $i,j$ either $m_i=n_j$ or $m_i \nmid n_j$ and $n_j \nmid m_i$.
Then the ideal 
\begin{equation*}
(\Theta_{m_1,\sred}\Theta_{m_2,\sred}\dotsm \Theta_{m_k,\sred}, \Theta_{n_1,\sred}\Theta_{n_2,\sred}\dotsm \Theta_{n_l,\sred})
\end{equation*}
in $A$ is principal.
\end{prop}

\begin{proof}
Let $I$ be the ideal above.
We may assume without loss of generality that $m_i \neq n_j$ for all $i,j$, i.e.~the generators of $I$ are coprime in $A$.
For each $i,j$ we can apply \cref{thetabezout} to obtain $a_{i,j},b_{i,j} \in A$ such that $a_{i,j}\Theta_{m_i,\sred}+b_{i,j}\Theta_{n_j,\sred}=1$.
Taking the product over all $i$ and $j$ we obtain
\begin{equation*}
1=\prod_{i} \left(\prod_{j} (a_{i,j}\Theta_{m_i,\sred} + b_{i,j}\Theta_{n_j,\sred})\right) \in \prod_{i} (\Theta_{m_i,\sred},\Theta_{n_1,\sred}\Theta_{n_2,\sred}\dotsm \Theta_{n_l,\sred}) \subseteq I
\end{equation*}
so $I=(1)$ is principal.
\end{proof}

For $f \in A$ and $l>1$ an integer, we define the cyclotomic valuation $\nu_{l,\sred}(f)$ to be the exponent of the highest power of $\Theta_{l,\sred}$ dividing $f$.
This extends to $\Frac A$ in the obvious way, namely we define $\nu_{l,\sred}(f/g)=\nu_{l,\sred}(f)-\nu_{l,\sred}(g)$ for $f,g \in A$.
If $f$ and $g$ are products of $\sred$-colored cyclotomic parts then
\begin{equation*}
\frac{f}{g}=\prod_l \Theta_{l,\sred}^{\nu_{l,\sred}(f/g)} \text{.}
\end{equation*}
For $f \in \overline{A}$ we similarly define $\nu_l(f)$ to be the highest power of $\Theta_l$ dividing $f$, which extends in a completely analogous way to $\Frac \overline{A}$.

\begin{lem} \label{valbinom}
Let $n,k$ be non-negative integers.
For all integers $1< l \leq n$ we have
\begin{equation*}
\nu_{l,\sred} \qbinom{n}{k}_{\sred}=\left\lfloor \frac{n}{l} \right\rfloor - \left\lfloor \frac{k}{l} \right\rfloor - \left\lfloor \frac{n-k}{l} \right\rfloor \text{.}
\end{equation*} 
In particular, $\nu_{l,\sred} \qbinom{n}{k}_{\sred} \in \{0,1\}$.
\end{lem}

\begin{proof}
Clearly
\begin{equation*}
\nu_{l,\sred}[m]_{\sred}=\begin{cases}
1 & \text{if $l|m$,} \\
0 & \text{otherwise,}
\end{cases}
\end{equation*}
so $\nu_{l,\sred}([m]_{\sred}!)=\lfloor m/l \rfloor$ and the equation above follows.
To show the bound on the valuation, note that $m/l-1<\lfloor m/l\rfloor \leq m/l$, so
\begin{multline*}
-1=\left(\frac{n}{l}-1\right)-\frac{k}{l}-\frac{n-k}{l}<\left\lfloor \frac{n}{l} \right\rfloor - \left\lfloor \frac{k}{l} \right\rfloor - \left\lfloor \frac{n-k}{l} \right\rfloor \\
<\frac{n}{l}-\left(\frac{k}{l}-1\right)-\left(\frac{n-k}{l}-1\right)=2 \text{.} \qedhere
\end{multline*}
\end{proof}

\begin{thm} \label{qbinomideal}
Let $n \in \NN$. 
The ideal
\begin{equation*}
\left(\qbinom{n}{1}_{\sred}, \qbinom{n}{2}_{\sred}, \dotsc, \qbinom{n}{n-1}_{\sred}\right)
\end{equation*}
in $A$ is principal, generated by $\Theta_{n,\sred}$.
\end{thm}

\begin{proof}
We will prove the result by induction.
Let
\begin{equation*}
I_m=\left(\qbinom{n}{1}_{\sred},\qbinom{n}{2}_{\sred}, \dotsc, \qbinom{n}{m}_{\sred}\right)
\end{equation*}
and write 
\begin{equation*}
[n]_{\sred}^{>m}=\prod_{\substack{k>m\\ k|n}} \Theta_{k,\sred} \text{.}
\end{equation*}
Suppose we have shown that $I_m$ is principal, generated by $[n]_{\sred}^{>m}$.
We will show that $I_{m+1}$ is principal, generated by $[n]_{\sred}^{>m+1}$.
It is enough to show that
\begin{equation}
\left(\qbinom{n}{m+1}_{\sred},[n]_{\sred}^{>m}\right)=([n]_{\sred}^{>m+1}) \text{.} \label{eq:qbinomideal-induction}
\end{equation}

Clearly $[n]_{\sred}^{>m+1}$ divides $[n]_{\sred}^{>m}$.
If $k>m+1$ and $k|n$ it is easy to see that
\begin{equation*}
\left\lfloor\frac{n}{k}\right\rfloor - \left\lfloor\frac{m+1}{k}\right\rfloor - \left\lfloor\frac{n-(m+1)}{k}\right\rfloor=1 \text{,}
\end{equation*}
so $\Theta_k$ divides $\qbinom{n}{m+1}_{\sred}$ exactly once, and thus $[n]_{\sred}^{>m+1}$ divides $\qbinom{n}{m+1}_{\sred}$.
If $m+1 \nmid n$ then $[n]_{\sred}^{>m+1}=[n]_{\sred}^{>m}$, \eqref{eq:qbinomideal-induction} follows trivially.

Otherwise suppose $m+1|n$. 
We claim that if $\Theta_{l,\sred}$ divides $\qbinom{n}{m+1}/[n]_{\sred}^{>m+1}$ we must have $l \nmid m+1$ and $m+1 \nmid l$.
This implies that 
\begin{equation*}
\left(\frac{\qbinom{n}{m+1}_{\sred}}{[n]_{\sred}^{>m+1}},\Theta_{m+1,\sred}\right)=(1)
\end{equation*}
by \cref{thetaprincipal}, from which \eqref{eq:qbinomideal-induction} holds and the result follows.

To prove the claim, suppose $l|m+1$.
It is straightforward to check that
\begin{equation*}
\left\lfloor\frac{n}{l}\right\rfloor - \left\lfloor\frac{m+1}{l}\right\rfloor - \left\lfloor\frac{n-(m+1)}{l}\right\rfloor=0 \text{,}
\end{equation*}
so $\Theta_{l,\sred}$ does not divide $\qbinom{n}{m+1}_{\sred}$, let alone $\qbinom{n}{m+1}_{\sred}/[n]_{\sred}^{>m+1}$.

Similarly, suppose $m+1|l$, and take $0 \leq r<l$ such that $n-(m+1)=ql+r$.
If $\Theta_{l,\sred}$ divides $\qbinom{n}{m+1}_{\sred}$ then 
\begin{equation*}
\left\lfloor\frac{n}{l}\right\rfloor - \left\lfloor\frac{m+1}{l}\right\rfloor - \left\lfloor\frac{n-(m+1)}{l}\right\rfloor=1
\end{equation*}
and we must have $r+m+1 \geq l$. 
Now let $d=\gcd(l,n-(m+1))$.
As $m+1|n$ and $m+1|l$, we have $m+1|d$ and in particular $m+1 \leq d$. 
We also have $d|r$, so in particular $l-r \geq d$.
We combine these two equalities to obtain $r+m+1 \leq l$, with equality if and only if $l-r=d$ and $m+1=d$.
This immediately implies that $l|n$, so $\Theta_{l,\sred}$ does not divide $\qbinom{n}{m+1}_{\sred}/[n]_{\sred}^{>m+1}$.
\end{proof}

\begin{thm} \label{qbinominvideal}
Let $n \in \NN$. 
The fractional ideal of $A$ generated by
\begin{equation*}
\qbinom{n}{0}_{\sred}^{-1}, \qbinom{n}{1}_{\sred}^{-1}, \dotsc, \qbinom{n}{n}_{\sred}^{-1}
\end{equation*}
is principal, generated by 
\begin{equation*}
\left(\prod_{\substack{1 \leq k \leq n\\ k \nmid n+1}} \Theta_{k,\sred}\right)^{-1} \text{.}
\end{equation*}
\end{thm}

\begin{proof}
We follow a similar strategy as in the proof of \cref{qbinomideal}.
Let $I_m$ denote the fractional ideal generated by
\begin{equation*}
\qbinom{n}{0}_{\sred}^{-1},\qbinom{n}{1}_{\sred}^{-1}, \dotsc, \qbinom{n}{m}_{\sred}^{-1}
\end{equation*}
and let 
\begin{equation*}
g_m=\prod_{\substack{k|n-m+i \text{ for some } 1 \leq i \leq m\\ k\nmid n+1}} \Theta_{k,\sred} \text{.}
\end{equation*}
Suppose we have shown that $I_m$ is principal, generated by $g_m^{-1}$.
We will show that $I_{m+1}$ is principal, generated by $g_{m+1}^{-1}$.
It is enough to show that the fractional ideal generated by
\begin{equation*}
\qbinom{n}{m+1}_{\sred}^{-1},g_m^{-1}
\end{equation*}
is equal to the principal fractional ideal generated by $g_{m+1}^{-1}$.
This is equivalent to proving equality of the following (ordinary) ideals
\begin{equation}
\left(g_m,\qbinom{n}{m+1}_{\sred}\right)=\left(\frac{\qbinom{n}{m+1}_{\sred}}{h_m}\right) \label{eq:qbinominvideal-induction}
\end{equation}
of $A$, where 
\begin{equation*}
h_m=\frac{g_{m+1}}{g_m}=\prod_{\substack{k|n-m\\ k \nmid n-m+1, k \nmid n-m+2, \dotsc ,k \nmid n+1}} \Theta_{k,\sred} \text{.}
\end{equation*}
(In particular, this shows that $\qbinom{n}{m+1}_{\sred}$ divides $g_m h_m=g_{m+1}$.)

We first check that the ideal on the right-hand side of \eqref{eq:qbinominvideal-induction} is an ordinary ideal. 
If $\Theta_{k,\sred}$ divides $h_m$ (i.e.~if $k|n-m$ and $k \nmid n-m+i$ for all $1 \leq i \leq m+1$) then $k \nmid m+1$ and the fractional part of $(n-(m+1))/k$ is $(k-1)/k$.
This implies that
\begin{equation*}
\left\lfloor \frac{n}{k} \right\rfloor - \left\lfloor \frac{m+1}{k} \right\rfloor - \left\lfloor \frac{n-(m+1)}{k} \right\rfloor = 1
\end{equation*}
so $\Theta_{k,\sred}$ also divides $\qbinom{n}{m+1}_{\sred}$.

It is clear that $\qbinom{n}{m+1}_{\sred}/h_m$ divides $\qbinom{n}{m+1}_{\sred}$. 
Suppose $\Theta_{k,\sred}$ divides $\qbinom{n}{m+1}_{\sred}/h_m$. 
Since we can write
\begin{equation*}
\qbinom{n}{m+1}=\frac{[n]_{\sred} [n-1]_{\sred} \dotsm [n-m]_{\sred}}{[m+1]_{\sred} [m]_{\sred} \dotsm [1]_{\sred}} \text{,}
\end{equation*}
this implies that either $k\nmid n-m$ and $k|n-m+i$ for some $1 \leq i \leq m$, or $k|n-m$ and $k|n-m+i$ for some $1 \leq i \leq m+1$. 
In either case, it is easy to check that $k\nmid n+1$, for otherwise $n/k$ has fractional part $(k-1)/k$, so
\begin{equation*}
\left\lfloor \frac{n}{k} \right\rfloor - \left\lfloor \frac{m+1}{k} \right\rfloor - \left\lfloor \frac{n-(m+1)}{k} \right\rfloor = 0
\end{equation*}
and $\Theta_{k,\sred}$ cannot divide $\qbinom{n}{m+1}_{\sred}$. 
This shows that $\qbinom{n}{m+1}_{\sred}/h_m$ divides $g_m$.

We will now show that
\begin{equation*}
\left(\frac{g_{m+1}}{\qbinom{n}{m+1}_{\sred}},h_m\right)=(1)
\end{equation*}
using \cref{thetaprincipal}, from which \eqref{eq:qbinominvideal-induction} holds and the result follows. 
It is enough to show that for any $l,d>1$, we do not have $\Theta_{l,\sred}|g_{m+1}/\qbinom{n}{m+1}_{\sred}$ and $\Theta_{ld,\sred}|h_m$, or $\Theta_{ld,\sred}|g_{m+1}/\qbinom{n}{m+1}_{\sred}$ and $\Theta_{l,\sred}|h_m$.

Suppose first that $\Theta_{l,\sred}|g_{m+1}/\qbinom{n}{m+1}_{\sred}$ and $\Theta_{ld,\sred}|h_m$. 
Then $l \nmid n+1$ and $ld|n-m$, so $l|n-m$ and $l \nmid m+1$. 
This shows that the fractional part of $(n-(m+1))/l$ is $(l-1)/l$ and the fractional part of $(m+1)/l$ is non-zero, so
\begin{equation*}
\left\lfloor \frac{n}{l} \right\rfloor - \left\lfloor \frac{m+1}{l} \right\rfloor - \left\lfloor \frac{n-(m+1)}{l} \right\rfloor = 1
\end{equation*}
which contradicts $\Theta_{l,\sred}|g_{m+1}/\qbinom{n}{m+1}_{\sred}$.

Similarly, suppose that $\Theta_{ld,\sred}|g_{m+1}/\qbinom{n}{m+1}_{\sred}$ and $\Theta_{l,\sred}|h_m$. 
Then $l|n-m$ and $l \nmid n-m+i$ for $1 \leq i \leq m+1$, while $ld|n-m+i$ for some $0 \leq i \leq m$ and $ld \nmid n+1$.
The only way this can happen is if $ld|n-m$. 
This implies that $ld \nmid m+1$, and we similarly obtain
\begin{equation*}
\left\lfloor \frac{n}{ld} \right\rfloor - \left\lfloor \frac{m+1}{ld} \right\rfloor - \left\lfloor \frac{n-(m+1)}{ld} \right\rfloor = 1
\end{equation*}
which contradicts $\Theta_{ld,\sred}|g_{m+1}/\qbinom{n}{m+1}_{\sred}$.
\end{proof}

\section{Existence and rotatability}

Let $Q=\Frac A$ and $\overline{Q}=\Frac \overline{A}$.
Our goal in this section is to prove \cref{existence} by showing that the denominators of the coefficients of $\JW_Q(\prescript{}{\sred}{n})$ divide 
\begin{equation*}
\prod_{\substack{1 \leq k \leq n\\ k \nmid n+1}} \Theta_{k,\sred}
\end{equation*}
by comparing them with the coefficients of $\JW_{\overline{Q}}(n)$.
First, we prove the analogous statement for $\JW_{\overline{Q}}(n)$.

\begin{lem} \label{onecolvalbound}
Let $k>1$ be an integer, and let $D$ be a one-colored Temperley--Lieb diagram in $\TL_{\overline{Q}}(n)$.
Then
\begin{equation*}
\nu_k\left(\coeff_{{}\in \JW_{\overline{Q}}(n)} D \right) \geq -1 \text{,}
\end{equation*}
with equality only if $1<k \leq n$ and $k \nmid n+1$.
\end{lem}

\begin{proof}
We proceed by induction. 
Suppose the result holds for $n=m$, and let $D$ be a one-colored Temperley--Lieb diagram in $\TL_{\overline{Q}}(m+1)$.
By the one-color version of \cref{JWcoef} we have
\begin{equation}
\begin{split}
\nu_k\left(\coeff_{{}\in \JW_{\overline{Q}}(m+1)} D \right)& =\nu_k\left(\sum_{\{i\}} \frac{[i]}{[m+1]}\coeff_{{}\in \JW_{\overline{Q}}(m)} D_i\right) \\
& \geq \min_{\{i\}} \left(\nu_k\left(\frac{[i]}{[m+1]}\right)+\nu_k\left(\coeff_{{}\in \JW_{\overline{Q}}(m)} D_i\right)\right) \text{.} 
\end{split}\label{eq:onecolvalbound}
\end{equation}
If $k \nmid m+1$, then $\nu_k([i]/[m+1]) \geq 0$ for any $i$ and $\nu_k(\coeff_{{}\in \JW_{\overline{Q}}(m)} D_i) \geq -1$.
On the other hand, if $k|m+1$, then $\nu_k([i]/[m+1]) \geq -1$ while $\nu_k(\coeff_{{}\in \JW_{\overline{Q}}(m)} D_i) \geq 0$.
In either case, the sum of the two valuations is at least $-1$, so the right-hand side of \eqref{eq:onecolvalbound} is at least $-1$.

Now suppose we have equality. 
By \cref{existence-onecol} the one-color Jones--Wenzl projector exists over the subring
\begin{equation*}
\overline{Q}_{\rm binom}=\ZZ[x]\left[\qbinom{m+1}{0}^{-1},\qbinom{m+1}{1}^{-1},\dotsc,\qbinom{m+1}{m+1}^{-1}\right] \text{.}
\end{equation*}
The natural embedding $\overline{Q}_{\rm binom} \subset \overline{Q}$ induces an embedding $\TL_{\overline{Q}_{\rm binom}}(m+1) \rightarrow \TL_{\overline{Q}}(m+1)$, and the image of $\JW_{\overline{Q}_{\rm binom}}(m+1)$ in $\TL_{\overline{Q}}(m+1)$ is clearly a Jones--Wenzl projector.
Since Jones--Wenzl projectors are unique, we conclude that the coefficients of $\JW_{\overline{Q}}(m+1)$ lie in $\overline{Q}_{\rm binom}$. 
In particular, if the $k$-valuation of any given coefficient is negative, then $\Theta_k$ must divide $\qbinom{m+1}{r}$ for some $0 \leq r \leq m+1$, so it must divide the least common multiple of $\qbinom{m+1}{0},\qbinom{m+1}{1},\dotsc,\qbinom{m+1}{m+1}$.
But a consequence of the one-color version of \cref{qbinominvideal} is that this least common multiple is
\begin{equation*}
g_{m+1}=\prod_{\substack{1 \leq r \leq m+1\\ r \nmid m+2}} \Theta_{r} \text{.}
\end{equation*}
So we must have $(\Theta_k,g_{m+1})=(\Theta_k)$, so $1 < k \leq m+1$ and $k \nmid m+2$ as required.
\end{proof}

Now let $A'=A[x]/(x^2-x_{\sred} x_{\tblu})$.
We view $A'$ as both an $A$-algebra and an $\overline{A}$-algebra in the obvious way. 
Writing $Q'=\Frac A'$, we have an isomorphism
\begin{equation*}
\begin{aligned}
\TL_{Q'}(n)& \longrightarrow 2\TL_{Q'}(\prescript{}{\sred}{n})\\
e_i & \longmapsto \begin{cases}
\frac{x}{x_{\sred}} e_i & \text{$i$ odd,} \\
\frac{x}{x_{\tblu}} e_i & \text{$i$ even,}
\end{cases}
\end{aligned}
\end{equation*}
which maps $\JW_{Q'}(n) \mapsto \JW_{Q'}(\prescript{}{\sred}{n})$.
So for any two-colored Temperley--Lieb diagram $D$, we have
\begin{equation*}
\coeff_{{}\in \JW_{Q}(\prescript{}{\sred}{n})} D=\coeff_{{}\in \JW_{Q'}(\prescript{}{\sred}{n})} D=x^a x_{\sred}^b x_{\tblu}^c \coeff_{{}\in \JW_{Q'}(n)} \overline{D}=x^a x_{\sred}^b x_{\tblu}^c \coeff_{{}\in \JW_{\overline{Q}}(n)} \overline{D}
\end{equation*}
for some integers $a,b,c$ for which $a+b+c=0$, where $\overline{D}$ denotes the one-color diagram obtained from $D$ by forgetting the coloring. 
It follows that when $k>2$ we have
\begin{align}
\nu_{k,\sred}\left(\coeff_{{}\in \JW_{Q}(\prescript{}{\sred}{n})} D \right)=\nu_{k,\tblu}\left(\coeff_{{}\in \JW_{Q}(\prescript{}{\sred}{n})} D \right)& =\nu_k\left(\coeff_{{}\in \JW_{\overline{Q}}(n)} \overline{D} \right) \text{,} \label{eq:valequalitynottwo} \\
\intertext{and}
\nu_{2,\sred}\left(\coeff_{{}\in \JW_{Q}(\prescript{}{\sred}{n})} D \right)+\nu_{2,\tblu}\left(\coeff_{{}\in \JW_{Q}(\prescript{}{\sred}{n})} D \right)& =\nu_k\left(\coeff_{{}\in \JW_{\overline{Q}}(n)} \overline{D} \right) \text{.} \label{eq:valequalitytwo}
\end{align}

\begin{lem} \label{twocolvalbound}
Let $k>1$ be an integer, and let $D$ be a two-colored Temperley--Lieb diagram in $2\TL_{Q}(\prescript{}{\sred}{n})$.
Then
\begin{equation*}
\nu_{k,\upur}\left(\coeff_{{}\in \JW_{Q}(\prescript{}{\sred}{n})} D \right) \geq -1 \text{,}
\end{equation*}
and if we have equality then $1 < k \leq n$ and $k \nmid n+1$, and $\upur=\sred$ if $k=2$.
\end{lem}

\begin{proof}
By \cref{onecolvalbound} and \eqref{eq:valequalitynottwo} we need only concern ourselves with the case where $k=2$. 
We proceed by induction as in the proof of \cref{onecolvalbound}.
Suppose the result holds for $n=m$, and let $D$ be a two-colored Temperley--Lieb diagram in $\TL_{Q}(\prescript{}{\sred}{(m+1)})$.
By \cref{JWcoef} we have
\begin{equation}
\begin{split}
\nu_{2,\upur}\left(\coeff_{{}\in \JW_{Q}(\prescript{}{\sred}{(m+1)})} D \right)& =\nu_{2,\upur}\left(\sum_{\{i\}} \frac{[i]_{\vgrn}}{[m+1]_{\sred}}\coeff_{{}\in \JW_{Q}(\prescript{}{\sred}{m})} D_i\right) \\
& \geq \min_{\{i\}} \left(\nu_{2,\upur}\left(\frac{[i]_{{\vgrn}}}{[m+1]_{\sred}}\right)+\nu_{2,\upur}\left(\coeff_{{}\in \JW_{Q}(\prescript{}{\sred}{m})} D_i\right)\right) \text{.} 
\end{split}\label{eq:twocolvalbound}
\end{equation}
If $m$ is even, then for all $i$
\begin{equation*}
\nu_{2,\upur}\left(\frac{[i]_{{\vgrn}}}{[m+1]_{\sred}}\right) \geq 0 \qquad \text{ and } \qquad \nu_{2,\upur}\left(\coeff_{{}\in \JW_{Q}(\prescript{}{\sred}{m})} D_i\right) \geq -1
\end{equation*}
by induction.
If $\upur=\tblu$, then for all $i$ 
\begin{equation*}
\nu_{2,\tblu}\left(\frac{[i]_{{\vgrn}}}{[m+1]_{\sred}}\right) \geq 0 \qquad \text{ and } \qquad \nu_{2,\tblu}\left(\coeff_{{}\in \JW_{Q}(\prescript{}{\sred}{m})} D_i\right) \geq 0
\end{equation*}
since $\nu_{2,\tblu}[m+1]_{\sred}=0$ for all $m$ and because the $\tblu$-colored valuation is non-negative by induction.
On the other hand, if $m$ is odd and $\upur=\sred$, then for all $i$ 
\begin{equation*}
\nu_{2,\upur}\left(\frac{[i]_{{\vgrn}}}{[m+1]_{\sred}}\right) \geq -1 \qquad \text{ and } \qquad \nu_{2,\sred}\left(\coeff_{{}\in \JW_{\overline{Q}}(\prescript{}{\sred}{m})} D_i\right) \geq 0 \text{.}
\end{equation*}
In all cases, the sum of the two valuations is at least $-1$ (and at least $0$ in the case where $\upur=\tblu$), so the right-hand side of \eqref{eq:twocolvalbound} is at least $-1$.

Now suppose $\upur=\sred$ and the left-hand side of \eqref{eq:twocolvalbound} is $-1$ but $m$ is even.
There is an involution of $\ZZ$-algebras (or a ``color-swap-twisted'' $R$-algebra involution)
\begin{equation*}
\begin{aligned}
\tau: 2\TL_{Q}(\prescript{}{\sred}{(m+1)};[2]_{\sred},[2]_{\tblu})& \longrightarrow 2\TL_{Q}(\prescript{}{\sred}{(m+1)};[2]_{\sred},[2]_{\tblu})\\
[2]_{\sred} & \longmapsto [2]_{\tblu} \\
[2]_{\tblu} & \longmapsto [2]_{\sred} \\
e_i & \longmapsto e_{(m+1)-i}
\end{aligned}
\end{equation*}
For a diagram $D$, $\tau(D)$ is the diagram obtained by reflecting $D$ about a vertical axis and swapping colors. 
Clearly this involution fixes $\JW_Q(\prescript{}{\sred}{n})$, so we have 
\begin{equation*}
\nu_{2,\sred}\left(\coeff_{{}\in \JW_{Q}(\prescript{}{\sred}{(m+1)})} D \right)=\nu_{2,\tblu}\left(\coeff_{{}\in \JW_{Q}(\prescript{}{\sred}{(m+1)})} \tau(D) \right) \geq 0
\end{equation*}
which is a contradiction, and completes the proof.
\end{proof}

\begin{lem} \label{qbinominvcoef}
Let $n,k$ be integers with $0 \leq k \leq n$.
There exists a two-colored diagram $D$ such that $\coeff_{{}\in \JW_Q(\prescript{}{\sred}{n})} D=\qbinom{n}{k}_{\sred}^{-1}$.
\end{lem}

\begin{proof}
Take $D$ to be the diagram with $k$ nested caps on the bottom left, $k$ nested cups on the top right, and all other strands connected from bottom to top. 
For example, if $n=5$ and $k=2$ we set
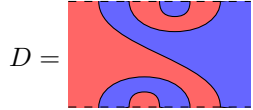
\begin{equation*}
D=\begin{gathered}
        \begin{tikzpicture}[xscale=.4,yscale=.2]
          \begin{scope}
            \clip (-3,3.5) rectangle (3,-3.5);
            \draw[fill=dgrmred] (2.5,4) rectangle (-4,-4);
            \draw[fill=dgrmblu] (4,-4) to (2,-3.5) to[out=90,in=-90] (-2,3.5) to (-2,4) to (-1,3.5) to[out=-90,in=180] (0.5,1.5) to[out=0,in=-90] (2,3.5) to (4,4) to (4,-4);
          \draw[fill=dgrmblu] (-2,-4) to (-2,-3.5) to[out=90,in=180] (-.5,-1.5) to[out=0,in=90] (1,-3.5) to (1,-4) to (0,-3.5) to[out=90,in=0] (-.5,-2.5) to[out=180,in=90] (-1,-3.5) to (-2,-4);
          \draw[fill=dgrmblu] (0,4) to (0,3.5) to[out=-90,in=180] (.5,2.5) to[out=0,in=-90] (1,3.5) to (0,4);          
       \end{scope}
       \draw[dashed] (-3,3.5) to (3,3.5);
       \draw[dashed] (-3,-3.5) to (3,-3.5);
     \end{tikzpicture}
\end{gathered}
\end{equation*}
The result follows by \cref{JWcoef} and induction on $n$.
\end{proof}

\begin{proof}[Proof of \cref{existence}]
Let
\begin{equation*}
T_R=\{f \in 2\TL_R(\prescript{}{\sred}{n}) : e_i f=0 \text{ for all } 1 \leq i \leq n-1\} \text{.}
\end{equation*}
In other words, $T_R$ is the (right) annihilator of the generators $e_1,\dotsc,e_{n-1}$. 
One can show that $\JW_R(\prescript{}{\sred}{n})$ exists if and only if there exists $f \in T_R$ for which $\coeff_{{} \in f} 1$ is invertible in $R$ (see e.g.~\cite[Exercise~9.25]{emtw} for the one-colored case). 
When this happens, $T_R=R\JW_R(\prescript{}{\sred}{n})$.

Clearly $\JW_Q(\prescript{}{\sred}{n})$ exists so $T_Q=Q\JW_Q(\prescript{}{\sred}{n})$. 
Thus $T_A$ is a free $A$-module of rank $1$, generated by $c\JW_Q(\prescript{}{\sred}{n}) \in \TL_A(\prescript{}{\sred}{n})$, where $c$ is the least common multiple of the denominators of the coefficients of $\JW_Q(\prescript{}{\sred}{n})$.
\Cref{twocolvalbound} implies that $c$ divides 
\begin{equation*}
g_n=\prod_{\substack{1 \leq k \leq n\\ k \nmid n+1}} \Theta_{k,\sred} \text{,}
\end{equation*}
while \cref{qbinominvcoef} and \cref{qbinominvideal} give $c=g_n$. 

Suppose $\qbinom{n}{k}_{\sred}$ is invertible in $R$ for all $0 \leq k \leq n$. 
Then $g_n$ is invertible in $R$ too by \cref{qbinominvideal}. 
Thus $T_R \geq R \otimes_A g_n \JW_Q(\prescript{}{\sred}{n})$ contains an element $f=1 \otimes g_n \JW_Q(\prescript{}{\sred}{n})$ for which $\coeff_{{} \in f} 1=g_n$ is invertible, so $\JW_R(\prescript{}{\sred}{n})$ exists.

Conversely, suppose $\JW_R(\prescript{}{\sred}{n})$ exists. 
We have $T_R=R\JW_R(\prescript{}{\sred}{n}) \geq R \otimes_A g_n\JW_Q(\prescript{}{\sred}{n})$, and thus $g_n\JW_R(\prescript{}{\sred}{n})=1 \otimes_A g_n\JW_Q(\prescript{}{\sred}{n})$.
But the coefficients of $g_n\JW_Q(\prescript{}{\sred}{n})$ (which lie in $A$) generate $(1)$ as an ideal of $A$ (again by \cref{qbinominvideal} and \cref{qbinominvcoef}), which directly implies that $g_n$ is invertible in $R$, and $\JW_R(\prescript{}{\sred}{n})=g_n^{-1} \otimes_A g_n \JW_Q(\prescript{}{\sred}{n})$.
\end{proof}

A consequence of the above computation is the aforementioned ``generic computation'' of coefficients of $\JW_R(\prescript{}{\sred}{n})$.

\begin{cor} \label{genericcoefs}
For each two-colored Temperley--Lieb diagram $D$ there are coprime elements $f_D,g_D \in A$ such that if $\JW_R(\prescript{}{\sred}{n})$ exists, the specialization of $g_D$ in $R$ is invertible for all $D$ and
\begin{equation*}
\coeff_{{} \in \JW_R(\prescript{}{\sred}{n})} D
\end{equation*}
is the specialization of $f_D/g_D$ in $R$. 
\end{cor}

\begin{rem}
We consider generic computation of the coefficients of one-colored Jones--Wenzl projectors (at least for subrings of $\CC$) to be mathematical folklore, i.e.~a ``known'' result without a published proof. 
In \cite[Theorem~6.13]{ew-localizedcalc} Elias--Williamson carefully prove an analogous result under the assumption that $R$ is both an integral domain and a henselian local ring.
Our proof does not require any restrictions on $R$ but is essentially equivalent to \cref{existence}.
\end{rem}

\begin{proof}
From the proof of \cref{existence}, if $\JW_R(\prescript{}{\sred}{n})$ exists then $g_n$ is invertible in $R$, and 
\begin{equation*}
\coeff_{{} \in \JW_R(\prescript{}{\sred}{n})} D = \coeff_{{} \in g_n^{-1} \otimes_A g_n\JW_Q(\prescript{}{\sred}{n})} D \text{.}
\end{equation*}
Set $f=\coeff_{{} \in g_n\JW_Q(\prescript{}{\sred}{n})}$, and take $f_D=f/\gcd(f,g_n)$ and $g_D=\gcd(f,g_n)$.
\end{proof}


For $f \in Q$, say that $f$ exists in $R$ if there are $a,b \in A$ with $f=a/b$ and $b$ invertible in $R$.

\begin{lem}
Suppose $\JW_R(\prescript{}{\sred}{n})$ exists. 
Then $\frac{[n+1]_{\sred}}{[k]_{\sred}}$ exists in $R$ for any integer $1 \leq k \leq n+1$.
\end{lem}

\begin{proof}
We have
\begin{equation*}
\frac{[n+1]_{\sred}}{[k]_{\sred}}=\frac{\prod_{l|n+1} \Theta_{l,\sred}}{\prod_{l|k} \Theta_{l,\sred}}=\frac{\prod_{\substack{l|n+1\\ l \nmid k}} \Theta_{l,\sred}}{\prod_{\substack{l|k\\ l \nmid n+1}} \Theta_{l,\sred}} \text{,}
\end{equation*}
and the denominator of the right-hand side divides
\begin{equation*}
\prod_{\substack{1 < l \leq n\\ l \nmid n+1}} \Theta_{l,\sred}
\end{equation*}
which is invertible by \cref{existence} and \cref{qbinominvideal}.
\end{proof}

\begin{prop} \label{simple-rotatability}
Suppose the two-colored Jones--Wenzl projectors $\JW_R(\prescript{}{\sred}{n})$ and $\JW_R(\prescript{}{\tblu}{n})$ exist.
Then $\JW_R(\prescript{}{\sred}{n})$ is rotatable if and only if $\frac{[n+1]_{\sred}}{[k]_{\sred}}=0$ for all integers $1 \leq k \leq n$.
\end{prop}

\begin{proof}
Calculating generically, we have
\begin{equation*}
\pTr(\JW_Q(\prescript{}{\sred}{n}))=-\frac{[n+1]_{\sred}}{[n]_{\sred}}\JW_Q(\prescript{}{\sred}{(n-1)})
\end{equation*}
by \eqref{eq:genericpTr}.
From the proof of \cref{genericcoefs} the coefficients of $\JW_Q(\prescript{}{\sred}{(n-1)})$ can be written as sums of fractions of the form $a\qbinom{n-1}{k}_{\sred}^{-1}$ for some $a \in A$ and some integer $0 \leq k \leq n-1$. 
Now observe that
\begin{equation*}
-\frac{[n+1]_{\sred}}{[n]_{\sred}}\frac{a}{\qbinom{n-1}{k}_{\sred}}=-\frac{[n+1]_{\sred}[k]_{\sred}! a}{[n]_{\sred}[n-1]_{\sred} \dotsm [n-k]_{\sred}}=-\frac{[n+1]_{\sred}}{[k+1]_{\sred}}\frac{a}{\qbinom{n}{k+1}_{\sred}}
\end{equation*}
noting that since $\JW_R(\prescript{}{\sred}{n})$ exists, $\qbinom{n}{k+1}_{\sred}$ is invertible. 
Thus $\JW_R(\prescript{}{\sred}{n})$ is rotatable if $\frac{[n+1]_{\sred}}{[k+1]_{\sred}}=0$.
Conversely, by \cref{qbinominvcoef} there is a a diagram whose coefficient in $\JW_Q(\prescript{}{\sred}{(n-1)})$ is exactly $\qbinom{n-1}{k}_{\sred}^{-1}$, so the above calculation shows that rotatability implies $\frac{[n+1]_{\sred}}{[k+1]_{\sred}}=0$.
\end{proof}

\begin{proof}[Proof of \cref{existsrotates}]
The condition on quantum binomial coefficients is the same as \cite[Assumption~1.1]{abe-homBS}. 
By \cite[Proposition~3.4]{abe-homBS} this implies that the quantum binomial coefficients $\qbinom{n}{k}_{\sred}$ and $\qbinom{n}{k}_{\tblu}$ are all invertible. 
Since
\begin{equation}
\qbinom{n+1}{k}_{\sred}=\frac{[n+1]_{\sred}}{[k]_{\sred}}\qbinom{n}{k-1}_{\sred} \label{eq:binomcoefincrement}
\end{equation}
and similarly for $\tblu$, we conclude that $\frac{[n+1]_{\sred}}{[k]_{\sred}}=\frac{[n+1]_{\tblu}}{[k]_{\tblu}}=0$ for all integers $1 \leq k \leq n$. 
Conversely, if the two-colored Jones--Wenzl projectors exist and are rotatable, then \eqref{eq:binomcoefincrement} combined with \cref{simple-rotatability} shows that $\qbinom{n+1}{k}_{\sred}$ and $\qbinom{n+1}{k}_{\tblu}$ vanish for all integers $1 \leq k \leq n$.
\end{proof}

\section{Applications to the Hecke category}

The diagrammatic Hecke category $\mathcal{H}$ of Elias--Williamson is constructed from a reflection representation of a Coxeter group called a \defnemph{realization}.
For each finite parabolic dihedral subgroup they identify a corresponding two-colored Temperley--Lieb algebra, whose defining parameters depend on the realization \cite[\S 5.2]{ew-soergelcalc}.
In \cite[\S 5]{ew-localizedcalc} Elias--Williamson highlight some hidden assumptions about their realizations from \cite{ew-soergelcalc}.
Their most basic assumption (without which the diagrammatic Hecke category is not well defined) is that certain two-colored Jones--Wenzl projectors exist and are rotatable. 
For the benefit of future work we give a corrected definition of a realization (which we call an \defnemph{Abe realization}) that ensures the existence and rotatability of these Jones--Wenzl projectors.

\begin{defn} \label{realizcorrected}
Let $\Bbbk$ be an integral domain.
An \defnemph{Abe realization} of a Coxeter system $(W,S)$ over $\Bbbk$ consists of a free, finite rank $\Bbbk$-module $V$ along with subsets 
\begin{align*}
\{\alpha_s : s \in S\} & \subset V &  \{\alpha_s^\vee : s \in S\} \subset V^\ast=\Hom_\field(V,\field)
\end{align*}
 such that
\begin{enumerate}[label=(\roman*)]
\item $\langle \alpha_s^\vee,\alpha_s \rangle=2$ for all $s \in S$;

\item the assignment 
\begin{equation*}
s(\beta)=\beta-\langle \alpha_s^\vee, \beta \rangle \alpha_s
\end{equation*} 
for all $s \in S$ and $\beta \in V$ defines a representation of the Coxeter group $W$ on $V$; 

\item \label{item:abecond} for all distinct $s,t \in S$ such that $st$ has order $m_{st}<\infty$, we have
\begin{equation*}
\qbinom{m_{st}}{k}_{\sred}(\langle \alpha_s^\vee,\alpha_t\rangle, \langle \alpha_t^\vee,\alpha_s\rangle)=\qbinom{m_{st}}{k}_{\tblu}(\langle \alpha_s^\vee,\alpha_t\rangle, \langle \alpha_t^\vee,\alpha_s\rangle)=0
\end{equation*}
for all integers $1 \leq k \leq m_{st}-1$.
\end{enumerate}
\end{defn}

By \cref{existsrotates}, condition \ref{item:abecond} above is equivalent to the existence and rotatability of $\JW_{\Bbbk}(\prescript{}{\sred}{(m_{st}-1)})$ and $\JW_{\Bbbk}(\prescript{}{\tblu}{(m_{st}-1)})$ for $[2]_{\sred}=\langle \alpha_s^\vee,\alpha_t\rangle$ and $[2]_{\tblu}=\langle \alpha_t^\vee,\alpha_s\rangle$. 
Thus \cref{Hwelldefined} follows from the discussion in \cite[\S 5.1]{ew-localizedcalc}. 
Moreover, condition \ref{item:abecond} is exactly Abe's assumption \cite[Assumption~1.1]{abe-homBS}, so \cref{abeequalsew} immediately follows by Abe's results \cite[Theorem~3.9]{abe-homBS} and \cite[Theorem~5.9]{abe-bimodhecke}.
It is also equivalent to 
\begin{equation} \label{eq:minpolycond}
\begin{aligned}
\Psi_{m_{st}}(\langle \alpha_s^\vee,\alpha_t \rangle \langle \alpha_t^\vee,\alpha_s \rangle)& =0 & &\text{if $m_{st}>2$,} \\
\langle \alpha_s^\vee,\alpha_t \rangle=\langle \alpha_t^\vee, \alpha_s^\vee \rangle& =0 & & \text{if $m_{st}=2$,}
\end{aligned}
\end{equation}
by \cref{qbinomideal}.

\begin{rem} \label{heckecategorification}
Abe's results assume that $\mathcal{H}$ categorifies the Hecke algebra. 
This only necessitates the additional assumption in \cref{abeequalsew} of \defnemph{Demazure surjectivity} \cite[Assumption~3.9]{ew-soergelcalc} for the realization $V$ \cite[\S\S 5.2--5.3]{ew-localizedcalc}.
This is a mild condition, and in particular holds if $2 \in \Bbbk^\times$. 
\end{rem}

\begin{rem} \label{H3defined}
There is a longstanding gap in the literature in defining the diagrammatic Hecke category for Coxeter groups containing a parabolic subgroup of type $H_3$. 
The diagrammatic Hecke category is currently not defined in such cases, because a crucial relation (the $H_3$ Zamolodchikov relation \cite[(5.12)]{ew-soergelcalc}) is incomplete.
One can argue that such a relation must exist in Abe's category, but explicitly determining this relation seems to be beyond current computational capabilities --- for further discussion see \cite[Remark~5.4]{ew-soergelcalc} and \cite[\S 3.6]{ew-localizedcalc}. 
Assuming such a relation can be found, it seems likely that \cref{Hwelldefined} would hold in this case as well.
\end{rem}

\begin{rem}
In \cite{ew-soergelcalc} Elias--Williamson incorrectly state that
\begin{equation} \label{eq:ewtechnicalcondition}
[m_{st}]_{\sred}(\langle \alpha_s^\vee,\alpha_t\rangle, \langle \alpha_t^\vee,\alpha_s\rangle)=[m_{st}]_{\tblu}(\langle \alpha_s^\vee,\alpha_t\rangle, \langle \alpha_t^\vee,\alpha_s\rangle)=0
\end{equation}
is enough to ensure the existence and rotatability of $\JW_\Bbbk(\prescript{}{\sred}{(m_{st}-1)})$.
(This error was identified in \cite{ew-localizedcalc} but only partially resolved there.)
In the same paper Elias--Williamson also incorrectly state that \eqref{eq:ewtechnicalcondition} is equivalent to \eqref{eq:minpolycond}.
Amusingly, when these two statements are combined these errors accidentally cancel and the resulting statement is equivalent to \cref{Hwelldefined}!
\end{rem}

\printbibliography

\end{document}